\documentclass[12pt,reqno]{amsart}
\usepackage{graphicx,subfigure}
\usepackage{hyperref}
\newcommand{\vertiii}[1]{{\left\vert\kern-0.25ex\left\vert\kern-0.25ex\left\vert #1
    \right\vert\kern-0.25ex\right\vert\kern-0.25ex\right\vert}}
\hypersetup{colorlinks=true, citecolor=blue, linkcolor=red}

\numberwithin{equation}{section} \numberwithin{figure}{section}
\numberwithin{table}{section} 
{ \theoremstyle{plain}
\newtheorem{theorem}{Theorem}[section]

\newtheorem{cor}[theorem]{Corollary}
\newtheorem{rem}[theorem]{Remark}

}

\begin{document}

\title[Energy estimates for minimizers to a class of elliptic systems]
{Energy estimates for minimizers to a class of elliptic systems of
Allen-Cahn type and the Liouville property}



\author{Christos Sourdis} \address{Department of Mathematics and Applied Mathematics, University of
Crete.}
              \email{csourdis@tem.uoc.gr}           %  \\


%\institute{F. Author \at
%              first address \\
%              Tel.: +123-45-678910\\
%              Fax: +123-45-678910\\
%              \email{fauthor@example.com}           %  \\
%%             \emph{Present address:} of F. Author  %  if needed


\maketitle

\begin{abstract}
We prove a theorem for the growth of the energy of bounded,
globally minimizing solutions to a class of semilinear elliptic systems of
the form $\Delta u=\nabla W(u)$, $x\in \mathbb{R}^n$, $n\geq 2$,
with $W:\mathbb{R}^m\to \mathbb{R}$, $m\geq 1$, nonnegative and
vanishing at exactly one point (at least in the closure of the image
of the considered solution $u$). As an application, we can prove a
Liouville type theorem under various assumptions.\end{abstract}
%By variational methods, we provide a simple proof of  existence of a
%heteroclinic orbit to the Hamiltonian system $u''=\nabla W(u)$ that
%connects the two global minima of a double-well potential $W$.
%In
%contrast to previous approaches, we do not impose any   condition
%near the global minima.
%Moreover, we consider several inhomogeneous extensions.
% Include keywords, PACS and mathematical
%subject classification numbers as needed.
%\keywords{variational methods \and heteroclinic connection \and
%Hamiltonian system}
%\end{abstract}

%\section{Introduction and statement of the main result}

\section{Introduction and statement of the main results}

Consider the semilinear elliptic system
\begin{equation}\label{eqEq}
\Delta u=\nabla W(u)\ \ \textrm{in}\ \ \mathbb{R}^n,\ \ n\geq 2,
\end{equation}
where $W:\mathbb{R}^m\to \mathbb{R}$, $m\geq 1$, is sufficiently
smooth and \emph{nonnegative}. This system has variational
structure, as solutions (in a smooth, bounded  domain $\Omega\subset
\mathbb{R}^n$) are critical points of the energy
\begin{equation}\label{eqEnergia}
E(v;\Omega)=\int_{\Omega}^{}\left\{\frac{1}{2}|\nabla v|^2+ W(v)
\right\}dx
\end{equation}
(subject to their own boundary conditions), where $|\nabla
v|^2=\sum_{i=1}^n |v_{x_i}|^2$. A solution $u\in
C^2(\mathbb{R}^n;\mathbb{R}^m)$ is called \emph{globally minimizing}
if
\begin{equation}\label{eqminimal}
E(u;\Omega)\leq E(u+\varphi;\Omega)
\end{equation}
for every smooth, bounded  domain $\Omega\subset \mathbb{R}^n$ and
for every $\varphi\in W^{1,2}_0(\Omega;\mathbb{R}^m)\cap
L^\infty(\Omega;\mathbb{R}^m)$ (see also \cite{fuscoCPAA} and the
references therein).

If $m\geq 2$, two are the main categories of such potentials $W$:
\begin{itemize}
  \item Those that vanish only on a discrete set of points (usually
  finite); in this case (\ref{eqEq}) is known as the vectorial
  Allen-Cahn equation and models multi-phase transitions (see \cite{baldo}, \cite{bronReih}, \cite{fonsecaTartarEdin} and the references that follow).
  \item Those that vanish on a continuum of points, as in the
  Ginzburg-Landau system (see \cite{bethuel}) or the elliptic system modeling
  phase-separation in \cite{berestykiARMA} or the one in \cite{caffareliLin}.
\end{itemize}
This article is motivated from the first class. In this setting, an
effective way to construct entire, nontrivial solutions to
(\ref{eqEq}) is to assume that $W$ is symmetric with respect to a
finite reflection group and to look for equivariant solutions. Under
proper assumptions, this roughly amounts to studying bounded,
globally minimizing solutions to (\ref{eqEq}) such that the closure
of their image contains at most one global minimum of $W$. In the
scalar case, that is $m=1$, this approach has been utilized, among
others, in \cite{cabre}. On the other side, recent progress has been made in the vector case in
\cite{alikakosARMA}, \cite{batesFuscoSmyrnelis}, and
\cite{fuscoPreprint}. In our oppinion, the main obstruction in the
  vector case is the lack of the maximum principle.
  This short discussion motivates our main result:

\begin{theorem}\label{thmMine}
Assume that $W\in C^1(\mathbb{R}^m;\mathbb{R})$, $m\geq 1$, and that
there exists  $a\in \mathbb{R}^m$ such that
\begin{equation}\label{eqpos}
W>0\ \ \textrm{in}\ \ \mathbb{R}^m\setminus \{a\}\ \ \textrm{and}\ \
W(a)=0.
\end{equation}
%and that
% there exists small $r_0>0$ such that the functions
%\begin{equation}\label{eqmonot}
%r\mapsto W(a_i+r\nu),\ \ \textrm{where}\ \ \nu \in \mathbb{S}^{m-1},
%\ \ \textrm{are strictly increasing\ for}\ r\in (0,r_0),\ \
%i=1,\cdots,N.
%\end{equation}
%\begin{equation}\label{eqinf}\liminf_{|u|\to \infty}
%W(u)>0.\end{equation}

 If $u\in C^2(\mathbb{R}^n;\mathbb{R}^m)$, $n\geq 2$, is
a bounded, globally minimizing solution to the elliptic system
(\ref{eqEq}), we have that \begin{equation}\label{eqdesir}\lim_{R\to
\infty}\left(\frac{1}{R^{n-1}} \int_{B_R}^{}\left\{\frac{1}{2}|\nabla
u|^2+ W(u) \right\}dx\right)=0,
\end{equation}
where $B_R$ stands for the $n$-dimensional ball of radius $R$ and
center at the origin.
%Additionally, if $W$ is merely $C^1(\mathbb{R}^m; \mathbb{R})$
%instead of $C^2(\mathbb{R}^m; \mathbb{R})$, we have that
%\[\lim_{R\to \infty}\frac{1}{R^{n-1}} \int_{B_R}^{}W(u)
%dx=0.
%\]
\end{theorem}

We emphasize that there is no assumption for the behavior of $W$
near $a$. To the best of our knowledge, besides the ordinary differential equation case $n=1$ (see
\cite{rabi}, \cite{Sourheteroclinic}), this is \emph{the first nontrivial result for the
vector case in this generality}.

Our proof of Theorem \ref{thmMine} is based on an adaptation to this
setting of the famous ``bad discs'' construction of \cite{bethuel}
from the study of vortices in the Ginzburg-Landau model. We stress
that, to the best of our knowledge, \emph{this is the first
application of this powerful technique to the study of the vector
Allen-Cahn equation}.

Moreover, we can provide a quantitative version of Theorem
\ref{thmMine}.

\begin{theorem}\label{thmMYNIG}
Assume that $W\in C^2(\mathbb{R}^m;\mathbb{R})$, $m\geq 1$,
satisfies (\ref{eqpos}) and that
 there exist $C_{0},r_1>0$ and $q\geq 2$ such that
\begin{equation}\label{eqkatzour}
W(a+r\nu)\geq C_{0}r^q,\ \ \textrm{where}\ \ \nu \in
\mathbb{S}^{m-1}, \ \ \textrm{for}\ r\in (0,r_1).
\end{equation}

 If $u\in C^2(\mathbb{R}^n;\mathbb{R}^m)$, $n\geq 2$, is
a bounded, globally minimizing solution to the elliptic system
(\ref{eqEq}), given positive $\tau<\frac{2}{qn}$, there exists $C(\tau)>0$ such that
\[ \int_{B_R}^{}\left\{\frac{1}{2}|\nabla u|^2+ W(u)
\right\}dx\leq C(\tau)R^{n-1-\tau},\ \ R>0.
\]
\end{theorem}

As an application of Theorem \ref{thmMine}, we can prove the
following Liouville type theorem.

\begin{theorem}\label{thmLiouva}
Assume that $W$ and $u$ are as in Theorem \ref{thmMine}. Then, it
holds that
\[
u\equiv a,
\]
provided that one of the following additional conditions holds:
\begin{description}
  \item[(a)] $m=1$ and $W\in C^{1,1}_{loc}(\mathbb{R};\mathbb{R})$; or $m\geq
  1$ and $u$ is radially symmetric; or $m\geq
  1$ and Modica's gradient bound holds, that is
  \begin{equation}\label{eqmodica}
\frac{1}{2}|\nabla u|^2\leq W(u) \ \ \textrm{in}\ \ \mathbb{R}^n.
\end{equation}
  \item[(b)]$n=2$ and there exists small $r_0>0$ such
that the functions
\begin{equation}\label{eqmonot}
r\mapsto W(a+r\nu),\ \ \textrm{where}\ \ \nu \in \mathbb{S}^{m-1}, \
\ \textrm{are strictly increasing\ for}\ r\in (0,r_0];
\end{equation}
or $n=2$, $W\in C^{1,1}_{loc}(\mathbb{R}^m;\mathbb{R})$, and the
above functions are nondecreasing for $r \in (0,r_0]$; or $n=2$ and
$m=1$.
  \item[(c)] $W\in C^2(\mathbb{R}^m;\mathbb{R})$ and
  \begin{equation}\label{eqAF}
W_{uu}(a)\nu\cdot \nu >0\ \ \textrm{for all}\ \ \nu \in
\mathbb{S}^{m-1},
  \end{equation}
where $\cdot$ stands for the Euclidean inner product in
$\mathbb{R}^m$.
\end{description}
\end{theorem}

The above theorem was originally proven by different techniques in
\cite{fuscoCPAA} (see also the earlier paper \cite{fuscoPreprint}),
under the conditions that ${W\in C^2(\mathbb{R}^m;\mathbb{R})}$ and
$u$ satisfy the assumptions of Theorem \ref{thmMine}, and that the
functions in (\ref{eqmonot}) have a \emph{strictly positive second
derivative} in $(0,r_0)$. In particular, the approach of the latter references is
based on a quantitative refinement of the replacement lemmas in \cite{alikakosARMA} and \cite{fuscoTrans}.
If $W$  additionally satisfies the
stronger assumption (\ref{eqAF}), this theorem was recently re-proven in
\cite{AlikakosDensity} by extending to this setting the density estimates of \cite{cafaCordoba}. In
the aforementioned references, the Liouville type theorem was proven
by an application of a basic pointwise estimate. However, it is not
difficult to convince oneself that the opposite direction is also
possible, that is the pointwise estimate follows from the Liouville property (see also \cite{sourdisUnif} for this viewpoint). We note that the pointwise estimate is the one that is
directly applicable in relation to the discussion preceding Theorem
\ref{thmMine}. This pointwise estimate roughly says that if $W$ (as in Theorem \ref{thmMine}) is such
that the Liouville type theorem holds, then a globally minimizing
solution, defined in a sufficiently large ball (with the appropriate
modifications in the definition), and bounded independently of the
size of the ball, has to be close to $a$ in the ball of half the
radius.

  In
the scalar case, under the assumptions of  the first part of Case
(a) above, this theorem can also be proven by using radial barriers
as in \cite{sourdisUnif}. On the other side, \emph{this powerful, but
intrinsically scalar technique, does not seem to work under the
minimal assumptions of the last part of Case (b)}.

In our opinion, three are the main advantages of our approach:
Firstly, we can treat in a unified and coordinate way various
situations. Secondly, in our opinion, our approach is considerably
simpler than those in the aforementioned references. Lastly, to the
best of our knowledge, it provides the strongest available result
when $n=2$ for any $m\geq 1$, \emph{even for the extensively studied scalar case}. This may seem too restrictive at first, but keep in mind that
 the dimensions $n=2,3$ are the ones with physical interest. In
fact, the  majority of papers on the subject deals exclusively
with these dimensions (see  \cite{alamaGui}, \cite{alesioITE}, \cite{bronReih},
 \cite{saez} for $n=2$ and \cite{guiSchat} for $n=3$).

 The proof of Theorem
\ref{thmLiouva} is based on combining Theorem \ref{thmMine} with a
variety of diverse results that are available in the literature.

In the sequel, we will provide the proofs of our main results. We
will close the paper with two appendixes that are used in the proofs,
but are also of independent interest and contain new results.

\section{Proof of the main results}\label{secProofs}

\subsection{Proof of Theorem \ref{thmMine}}

\begin{proof}
Throughout this proof, we will denote the energy density of $u$ by
\begin{equation}\label{eqeee}
e(x)=\frac{1}{2}|\nabla u(x)|^2+W\left( u(x) \right),\ \ x\in
\mathbb{R}^n.
\end{equation}

Firstly, note that standard elliptic regularity theory and Sobolev
imbeddings  \cite{evans,Gilbarg-Trudinger}, in combination  with the
fact that $u$ is bounded, yield that
\begin{equation}\label{eqapriori}
\|u\|_{C^{1,\alpha}(\mathbb{R}^n;\mathbb{R}^m)}\leq C_1,
\end{equation}
for some $\alpha\in (0,1)$ and $C_1>0$ (in fact, it holds for any
$\alpha\in (0,1)$ provided that $C_1=C_1(\alpha)>0$). We point out
that this is the only place where we use that $W\in C^2$.

Since $u$ is a globally minimizing solution, by comparing its energy
to that of a suitable test function which agrees with $u$ on
$\partial B_R$ and is identically $a$ in $B_{R-1}$, we find that
\begin{equation}\label{eqAmbrosio21}
\int_{B_R}^{}e(x)dx\leq C_2 R^{n-1},\ \ R\geq 1,
\end{equation}
for some $C_2>0$ (see also \cite{cafaCordoba}).

Therefore, by (\ref{eqAmbrosio21}), the co-area formula (see for
instance \cite[Ap. C]{evans}), the nonnegativity of $W$, and the
mean value theorem, there exist
\begin{equation}\label{eqSR}
S_R\in \left(R,2R\right)
\end{equation}
such that
\begin{equation}\label{eqcoarea}
\int_{\partial B_{S_R}}^{} e(x)dS(x)\leq C_3R^{n-2},\ \ R\geq 1,
\end{equation}
for some $C_3>0$ (actually, we can take $C_3=\frac{C_2}{2}$).

Let $\epsilon>0$ be any small number. By virtue of
(\ref{eqapriori}),  we can infer that the subset of $\partial
B_{S_R}$ where $e(x)$ is above $\epsilon$ is contained in at most
$\mathcal{O}(R^{n-2})$ number of geodesic  balls of radius $1$ as
$R\to \infty$ (the so-called ``bad discs'', see \cite{bethuel}).
More precisely, there exist $N_{\epsilon,R}\geq 0$ points
$\{x_{R,1},\cdots,x_{R,N_{\epsilon,R}}\}$ on $\partial B_{S_R}$ such
that
\[
e(x)\geq \epsilon\ \ \textrm{if}\ \ x\in U_R(x_{R,i},1),\ \
i=1,\cdots,N_{\epsilon,R},
\]
and
\begin{equation}\label{eqgood}
e(x)\leq \epsilon\ \ \textrm{if}\ \   x\in \partial B_{S_R}
\setminus \bigcup_{i=1}^{N_{\epsilon,R}}U_R(x_{R,i}, 1),
\end{equation}
for  $R\gg 1$, where $U_R(p,r)\subset \partial B_{S_R}$ stands for
the geodesic ball with center at $p$ and radius $r$. Moreover, we
have that
\begin{equation}\label{eqN}
N_{\epsilon,R}\leq M_\epsilon R^{n-2},\ \ R\gg 1\ (\textrm{with}\
M_\epsilon>0\ \textrm{independent of} \ R).
\end{equation}

In the sequel, we will prove the above properties by adapting some
arguments from \cite{bethuel}. Firstly, we prove a
\emph{clearing-out property}. Note that (\ref{eqapriori}) implies
that there exists $\mu_\epsilon<\epsilon$ such that
\[
\int_{U_R(y,2)}^{}e(x)dS(x)<\mu_\epsilon\ \ \textrm{for some}\ y\in
\partial B_{S_R} \] implies that
\[e(x)\leq \epsilon,\ \ x\in U_R(y,1),
\]
for $R\geq 1$. Indeed, suppose that
\begin{equation}\label{eqexclude}e(z)\geq \epsilon\ \textrm{for
some} \ y\in \partial B_{S_R} \ \textrm{and}\ z\in
U_R(y,1).\end{equation} From (\ref{eqapriori}), there exists $C_4>0$
such that
\[\| e\|_{C^{0,\alpha}(\mathbb{R}^n;\mathbb{R})}\leq C_4.
\]
It then follows that
\[
e(x)\geq \epsilon-C_4 d^\alpha,\ \ x\in B(z,d)=z+B_d,
\]
for all  $d< \min
\left\{1,\left(\frac{\epsilon}{2C_4}\right)^\frac{1}{\alpha}
\right\}$ (see also \cite[Lem. 2.3]{sternbergZumb}). Since $e\geq
0$, we find that
\[
\int_{U_R(y,2)}^{}e(x)dS(x)\geq \int_{U_R(z,d)}^{}e(x)dS(x)\geq
(\epsilon-C_4 d^\alpha)\left|U_R(z,d) \right|\geq
\frac{\epsilon}{2}\left|U_R(z,d)
\right|=\frac{\epsilon}{2}|\mathbb{S}^{n-1}|d^{n-1}.
\]
Hence, we can exclude the scenario (\ref{eqexclude})  by choosing
\begin{equation}\label{eqmuuu}
\mu_\epsilon=\frac{\epsilon}{2^n}|\mathbb{S}^{n-1}|\left( \min
\left\{1,\left(\frac{\epsilon}{2C_4}\right)^\frac{1}{\alpha}\right\}
\right)^{n-1}.
\end{equation}
%Clearly, the last quantity does not depend on the specific $z\in \partial B_{S_j}$ and
%converges to a positive constant as $j\to \infty$.
Next, consider a finite family of geodesic  balls $U_R(x_i,1)_{i\in
I_R}$ such that
%\[x_i\in
%\partial B_{S_j}\ \  \forall i\in I,
%\]
\begin{equation}\label{eqcap}
U_R\left(x_i,\frac{1}{4}\right)\cap
U_R\left(x_k,\frac{1}{4}\right)=\emptyset\ \ \textrm{if}\ \ i\neq k,
\end{equation}
\begin{equation}\label{eqcover}
\bigcup_{i\in I_R} U_R(x_i,1)\supset \partial B_{S_R},
\end{equation}
for all $R\geq 1$ (having suppressed the obvious dependence of $x_i$
on $R$).
 This is indeed possible by the Vitali's covering theorem (see \cite[Sec. 1.5]{evansGariepy} and keep in mind
 that $\partial B_{S_R}$ becomes a metric space when equipped with the geodesic distance).
 We say that the ball $U_R(x_i,1)$ is a \textbf{good
ball} if
\[
\int_{U_R(x_i,2)}^{}e(x)dS(x)<\mu_\epsilon,
\]
and that $U_R(x_i,1)$ is a \textbf{bad ball} if
\[
\int_{U_R(x_i,2)}^{}e(x)dS(x)\geq \mu_\epsilon.
\]
The collection of bad balls is labeled by
\[
J_R=\left\{i\ :\ U_R(x_i,1)\ \textrm{is\ a\ bad\ ball} \right\}.
\]
The main observation is that, by virtue of (\ref{eqcap}), there is a
universal constant $C_5>0$ (independent of both $\epsilon$ and $R$)
such that
\[
\sum_{i\in I_R} \int_{U_R(x_i,2)}^{}e(x)dS(x)\leq C_5 \int_{\partial
B_{S_R}}^{}e(x)dS(x),
\]
since each point on $\partial B_{S_R}$ is covered by at most $C_5$
geodesic balls $U_R(x_i,2)$. The latter property plainly  follows by
observing that all such balls that contain the same point are
certainly contained in a geodisc ball of radius $10$, and from the
basic fact that any $(n-1)$-dimensional ball of radius $10$ can
contain only a finite number of disjoint balls of radius
$\frac{1}{4}$. Making use of (\ref{eqcoarea}), we infer that
\begin{equation}\label{eqcard}
\textrm{card}J_R\leq \frac{C_5C_3}{\mu_\epsilon} R^{n-2},\ \ R\geq
1.
\end{equation}
Now, let $x\in \partial B_{S_R}\setminus \bigcup_{i\in
J_R}U_R(x_i,1)$. By (\ref{eqcover}), there exists some $k\in
I_R\setminus J_R$ such that $x\in U_R(x_k,1)$ which is a good ball.
It follows from the definition of $\mu_\epsilon$ that
\[
e(x)\leq \epsilon,
\]
as desired.

In view of (\ref{eqpos}) and (\ref{eqgood}), we have that
\begin{equation}\label{eqmeps}
|\nabla u(x)|^2\leq 2\epsilon\ \ \textrm{and}\ \
\left|u(x)-a\right|\leq m_\epsilon\ \ \textrm{if}\ \ x\in \partial
B_{S_R}\setminus \bigcup_{i=1}^{N_{\epsilon,R}}U_R(x_{R,i}, 1),\ \
R\gg 1,
\end{equation}
where
\begin{equation}\label{eqmepsnew}
m_\varepsilon\to 0\ \ \textrm{as}\ \ \varepsilon\to 0,
\end{equation}
(we point out that $m_\epsilon$ depends only on $\epsilon$).

We consider the function $v_R\in
W^{1,2}\left(B_{S_R};\mathbb{R}^m\right)\cap
L^\infty\left(B_{S_R};\mathbb{R}^m \right)$ which is defined, in
terms of polar coordinates, as follows:
\[
v_R(r,\theta)=\left\{\begin{array}{ll}
                       u(S_R,\theta)+\left(a-u(S_R,\theta) \right)(S_R-r), & r\in [S_R-1,S_R],\ \theta\in \mathbb{S}^{n-1}, \\
                         &   \\
                       a, & r\in [0,S_R-1],\ \theta\in \mathbb{S}^{n-1},
                     \end{array}
 \right.
\]
(having slightly abused notation, keep in mind that $x=r\theta$). We note that $v_R$ belongs in $W^{1,2}$ because it is the composition of a smooth function with a Lipschitz continuous one (see \cite[pg. 54]{kinderler} and keep in mind that we only use the polar coordinates away from the origin).
Clearly, we have that
\begin{equation}\label{eqonbdry}
v_{R}=u \ \ \textrm{on}\ \ \partial B_{S_R}.
\end{equation}

Let
\[
\mathcal{A}_R=B_{S_R}\setminus B_{(S_R-1)}\ \ \textrm{and}\ \
\mathcal{C}_R=\bigcup_{i=1}^{N_{\epsilon,R}}\left(\bar{B}_{10}(x_{R,i})\cap
\bar{\mathcal{A}}_R \right).
\]
If $x=r\theta\in \mathcal{A}_R\setminus \mathcal{C}_R$, via
(\ref{eqmeps}), it holds that \begin{equation}\label{eqthan1}
\left|v_R(x)-a \right|\leq 2 \left|u(S_R,\theta)-a \right|\leq 2
m_\epsilon.
\end{equation}
Moreover, for such $x$, we find that
\begin{equation}\label{eqthan2}
\begin{array}{rcl}
  |\nabla_{\mathbb{R}^n} v_R|^2 & = & \left|u(S_R,\theta)-a \right|^2+\frac{1}{r^2}
\left|(1+r-S_R)\nabla_{\mathbb{S}^{n-1}}u(S_R,\theta) \right|^2
   \\
    &   &   \\
 \textrm{using}\  (\ref{eqSR}), (\ref{eqmeps}):  &\leq  & m_\epsilon^2+\frac{9}{S_R^2}
\left|\nabla_{\mathbb{S}^{n-1}}u(S_R,\theta) \right|^2 \\
    &   &   \\
    & \leq  & m_\epsilon^2+9
\left|\nabla_{\mathbb{R}^n}u(S_R\theta) \right|^2 \\
    &   &   \\
  \textrm{using again}\ (\ref{eqmeps}): & \leq & m_\epsilon^2+18\epsilon,
\end{array}
\end{equation}
where we made repeated use of the identity
%Moreover, using (\ref{eqapriori}) and (\ref{eqmeps}), it follows
%readily that
%\[
%\int_{B_{S_j}}^{}\left\{\frac{1}{2}|\nabla v_j|^2+ W(v_j)
%\right\}dx\leq C_4 (m_\epsilon+\epsilon)S_j^{n-1}+C_5
%N_{\epsilon,j},\ \ R\gg 1,
%\]
%where $C_4,C_5>0$ are independent of both $R$ and $\epsilon$.
\[
|\nabla_{\mathbb{R}^n} v|^2=|v_r|^2+\frac{1}{R^2}
|\nabla_{\mathbb{S}^{n-1}} v|^2\ \ \textrm{on}\ \partial B_R,\ \
R>0,
\]
(see \cite[Ch. 8]{taylor}). It follows that
\[
\begin{array}{rcl}
  \int_{B_{S_R}}^{}\left\{\frac{1}{2}|\nabla v_R|^2+ W(v_R)
\right\}dx & = & \int_{\mathcal{A}_R}^{}\left\{\frac{1}{2}|\nabla
v_R|^2+ W(v_R)
\right\}dx \\
    &   &   \\
\textrm{using (\ref{eqapriori})}: & \leq & C_6
N_{\epsilon,R}+\int_{\mathcal{A}_R\setminus\mathcal{C}_R}^{}\left\{\frac{1}{2}|\nabla
v_R|^2+ W(v_R)
\right\}dx \\
    &   &  \\
\textrm{using (\ref{eqthan1}), (\ref{eqthan2})}:    & \leq & C_6
N_{\epsilon,R}+
\left(\frac{m_\epsilon^2}{2}+9\epsilon+C_7m_\epsilon \right) |\mathcal{A}_R\setminus\mathcal{C}_R|  \\
     &   &   \\
    &\leq   &  C_6 N_{\epsilon,R}+C_8(m_\epsilon+\epsilon)S_R^{n-1},
\end{array}
\]
where $C_6,C_7,C_8>0$ are independent of both $\epsilon$ and $R$.
Since $u$ is a globally minimizing solution, thanks to
(\ref{eqonbdry}), we obtain that
\begin{equation}\label{equseful}
\begin{array}{rcl}
  \int_{B_{S_R}}^{}e(x)dx & \leq & C_6 N_{\epsilon,R}+C_8 (m_\epsilon+\epsilon)S_R^{n-1} \\
    &   &   \\
  \textrm{using}\ (\ref{eqSR}), (\ref{eqN}): & \leq &  2^{n-2}C_6M_\epsilon R^{n-2}+2^{n-1}C_8
(m_\epsilon+\epsilon)R^{n-1},\ \ R\gg 1.
\end{array}
 \end{equation}
Since $\epsilon>0$ is arbitrary, in light of (\ref{eqmepsnew}), we
infer that (\ref{eqdesir}) holds, as desired.
\end{proof}

\begin{rem}
The assertion  of Theorems \ref{thmLiouva}  holds \emph{for any bounded solution}  of (\ref{eqEq}) provided that $W$ is assumed to be globally convex (see for example \cite{sourConfi}).
\end{rem}

\subsection{Proof of Theorem \ref{thmMYNIG}}

\begin{proof}
The proof is based on a bootstrap argument.
%The main observation is the following:
Assume that
\[
\int_{B_R}^{}e(x)dx\leq C(k) R^{k}, \ R>0, \ \textrm{for some}\ k\in
(0,n-1],
\]
(recall the notation (\ref{eqeee})). In the proof of Theorem
\ref{thmMine}, choose
\[
\epsilon_R=R^{-\beta},
\]
for $R\gg 1$, with $\beta>0$ to be chosen. In analogy to
(\ref{eqSR})-(\ref{eqcoarea}), we now take $S_R\in (R,2R)$ such that
\[
\int_{\partial B_{S_R}}^{} e(x)dS(x)\leq \frac{1}{2}C(k)R^{k-1},\ \
R>0.
\]
Let $N_R$ be the minimal number of geodesic balls of radius $1$ on
$\partial B_{S_R}$ which contain the set where $e(x)$ is above
$\epsilon_R=R^{-\beta}$. In view of (\ref{eqmuuu}), the
corresponding quantity can be chosen so that it satisfies
\[
\mu_{\epsilon_R}\geq C_{11}\epsilon_R^n=C_{11}R^{-\beta n}
\]
for $R\gg 1$, where $C_{11}>0$. In turn, as in (\ref{eqcard}), we
obtain that
\[
N_R\leq C_{12} R^{k-1+\beta n},\ \ R\gg 1,
\]
where $C_{12}>0$. By virtue of (\ref{eqkatzour}), the corresponding
quantity in (\ref{eqmeps}) satisfies
\[
m_{\epsilon_R}\leq C_{13}
\epsilon_R^\frac{1}{q}=C_{13}R^{-\frac{\beta}{q}},\ \ R\gg 1,
\]
for some $C_{13}>0$. Substituting the above in the analog of
(\ref{equseful}), yields that
\[
\int_{B_{S_R}}e(x)dx\leq
C_{14}\left(R^{n-1-\frac{2\beta}{q}}+R^{k-1+\beta n} \right),\ \
R\gg 1,
\]
for some $C_{14}>0$. We will choose $\beta$ so that the two
exponents in the above relation are equal, this gives
\[
\beta=\frac{q(n-k)}{qn+2}.
\]
We have arrived at
\[
\int_{B_R}^{}e(x)dx\leq C_{15}R^{\gamma(k)},\ \ R>0,
\]
where
\[
\gamma(k)=n-1-\frac{2(n-k)}{qn+2}.
\]
To conclude, observe that the mapping $k\to \gamma(k)$ is a
contraction with $n-1-\frac{2}{qn}$ as a fixed point.
\end{proof}

\subsection{Proof of Theorem \ref{thmLiouva}}
\begin{proof}
\underline{Case (a)} If $u$ satisfies (\ref{eqmodica}), since $W\geq
0$, it is known that the following strong monotonicity formula
holds:
\begin{equation}\label{eqmonotoniWeak}
\frac{d}{dR}\left(\frac{1}{R^{n-1}}\int_{B_R}^{}\left\{\frac{1}{2}|\nabla
u|^2+ W\left(u\right) \right\}dx\right)\geq 0,\ \ R>0,
\end{equation}
(see \cite{cafamodica}, \cite{modicaProc} for $m=1$ and \cite{alikakosBasicFacts}, \cite{caffareliLin} for
arbitrary $m\geq 1$). We point out that the fact that $u$ is minimal
is not used for this. Hence, for any positive $r<R$, we have that
\[
\frac{1}{r^{n-1}}\int_{B_r}^{}\left\{\frac{1}{2}|\nabla u|^2+
W\left(u\right) \right\}dx \leq
\frac{1}{R^{n-1}}\int_{B_R}^{}\left\{\frac{1}{2}|\nabla u|^2+
W\left(u\right) \right\}dx.
\]
By virtue of Theorem \ref{thmMine}, letting $R\to \infty$ in the
above relation yields that $u\equiv a$.
 For the reader's convenience, we will present
the proof of a seemingly new monotonicity formula in Appendix
\ref{AppMonot} which can also be used to reach the same result.

To complete the proof in this case, we note that the gradient
estimate (\ref{eqmodica}) was shown in \cite{farinaFlat} to hold for
\emph{any} bounded, entire solution when $m=1$ and $W\in
C^{1,1}_{loc}(\mathbb{R};\mathbb{R})$ is nonnegative (see
\cite{cafamodica}, \cite{modica} for earlier proofs which required
higher regularity on $W$). Lastly, it is easy to show that any
radially symmetric solution satisfies this gradient bound for any
$m\geq 1$ and $W\in C^1$ nonnegative (see \cite{sourdis14}).

\underline{Case (b)} Here we partly follow \cite{sourdis14}. Since
$n=2$, by working as in (\ref{eqcoarea}), and using the assertion of
Theorem \ref{thmMine}, we arrive at
\[
\int_{\partial B_{S_R}}^{} W\left(u(x)\right)dS(x)\to 0,\ \
\textrm{for some}\ \ S_R\in (R,2R),\ \ \textrm{as}\ \  R\to \infty.
\]
By making use of just the $C^1$-bound in (\ref{eqapriori}), and
working as we did in order to exclude (\ref{eqexclude}), we deduce that
\begin{equation}\label{eqbdry}
\max_{|x|=S_R}\left|u(x)-a \right|\to 0\ \ \textrm{as}\ \ R\to
\infty.
\end{equation}

Under the assumptions of the first part of Case (b), a recent
variational maximum principle from \cite{alikakosPreprint}  implies
that
\[
\max_{|x|\leq S_R}\left|u(x)-a \right|\leq
\max_{|x|=S_R}\left|u(x)-a \right|,
\]
(see also Appendix \ref{appMAX} herein). Moreover, as we will prove
in Appendix \ref{appMAX},  this variational maximum principle also holds under
the assumptions of the second part of Case (b). In light of (\ref{eqbdry}), by letting $R\to \infty$ in the above relation, we can
conclude that the
first two assertions in Case (b) hold.

We will establish   the validity of the last assertion of Case (b)
by borrowing some ideas from \cite{villegasCpaa}, while adopting a
more explanatory viewpoint. To this end, we will argue by
contradiction. Without loss of generality, we may assume that there
exists a sequence $R_j\to \infty$ and a $\delta>0$ such that
\[
u(x_{j})=\max_{|x|\leq S_{R_j}}u(x)\geq a+\delta,\ \ j\geq 1,
\]
for some $x_{j}\in B_{S_{R_j}}$. In particular, there exists $d\in
(0,\delta)$ such that
\[
W(a+d)<W\left(u(x_{j}) \right),\ \ j\geq 1.
\]
By virtue of (\ref{eqbdry}), we may further assume that
\begin{equation}\label{eqvilleg2}
\max_{|x|=S_{R_j}}u(x)\leq a+\frac{d}{2},\ \ j\geq 1.
\end{equation}
Let $u_j\in \left[a+d,u(x_{j})\right)$ be such that
\begin{equation}\label{eqvilleg1}
W(u_j)=\min_{u\in\left[a+d,u(x_{j})\right]}W(u).
\end{equation}
Consider the competitor function
\[
V_j(x)=\min\left\{u(x),u_j \right\},\ \ x\in  B_{S_{R_j}},
\]
which belongs in $W^{1,2}\left(B_{S_{R_j}};\mathbb{R}^m\right)\cap
L^\infty\left(B_{S_{R_j}};\mathbb{R}^m \right)$ (see for instance
\cite{dancerPlanes}) and, thanks to (\ref{eqvilleg2}), agrees with
$u$ on $\partial B_{S_{R_j}}$. To conclude, we will  show that
\[
E\left(V_j;B_{S_{R_j}} \right)<E\left(u;B_{S_{R_j}} \right),
\]
which contradicts the minimality character of $u$.
 To this aim, let \[ \mathcal{D}_j=\left\{x\in B_{S_{R_j}}\ :\ u(x)>u_j
 \right\}.
 \]
Observe that $\mathcal{D}_j$ is nonempty (since it contains
 $x_j$) and strictly contained in $B_{S_{R_j}}$ (from (\ref{eqvilleg2})).
 Then, note that
 \[
 E\left(V_j;B_{S_{R_j}}\setminus \mathcal{D}_j\right)=
 E\left(u;B_{S_{R_j}}\setminus \mathcal{D}_j\right)
 \ \
\textrm{and}\ \
 E\left(V_j; \mathcal{D}_j\right)=E\left(u_j;\mathcal{D}_j\right)<E\left(u;\mathcal{D}_j\right),
 \]
since (\ref{eqvilleg1}) holds and there exists a connected component
$\mathcal{E}_j$ of $\mathcal{D}_j$, say the one containing $x_j$,
where $u$ is nonconstant (note that $u=u_j$ on $\partial \mathcal{D}_j$).
%such that
%\[
%\int_{\mathcal{E}_j}^{}|\nabla u|^2
%dx>0=\int_{\mathcal{E}_j}^{}|\nabla V_j|^2 dx.
%\]

 \underline{Case (c)} Let
\[
\varepsilon=\frac{1}{R}\ \ \textrm{and}\ \
u_\varepsilon(y)=u\left(\frac{y}{\varepsilon} \right),\ \ y\in
\mathbb{R}^n.
\]
We have that
%\begin{equation}\label{eqtonegawa}
\[
\varepsilon^2 \Delta u_\varepsilon =\nabla W(u_\varepsilon)\ \
\textrm{in}\ \ \mathbb{R}^n.
\]
%\end{equation}
Moreover, thanks to Theorem \ref{thmMine} and (\ref{eqAF}), we have
that
\[
\frac{1}{\varepsilon} \int_{B_2}^{}\left|u_\varepsilon(y)-a\right|^2
dy\to 0\ \ \textrm{as}\ \  \varepsilon \to 0.
\]
In particular, this implies that
\[
\|u_\varepsilon-a\|_{L^1 (B_2)}\to 0 \ \ \textrm{as}\ \  \varepsilon
\to 0.
\]
If $m=1$, we can obtain the uniform convergence
\[
\|u_\varepsilon-a\|_{L^\infty (B_1)}\to 0 \ \ \textrm{as}\ \
\varepsilon \to 0,
\]
by appealing to the results of \cite{cafaCordoba}.
 The latter results were recently extended to cover the case $m\geq
2$ in \cite{AlikakosDensity}. Therefore, we deduce that the above
uniform estimate holds under the assumptions of Case (c). We can
then conclude by switching back to $u$ and $R$ and letting $R\to
\infty$, similarly to Case (b).\end{proof}

\appendix

\section{A new monotonicity formula for solutions to the elliptic system $\Delta u=\nabla
W(u)$}\label{AppMonot} In this appendix, we will prove a seemingly
new monotonicity formula which can be used in the proof of the first
case of Theorem \ref{thmLiouva}.

\begin{theorem}\label{thmMonot}
If $u\in C^2(\mathbb{R}^n;\mathbb{R}^m)$, $n\geq 2, m\geq 1$, solves
(\ref{eqEq}) with $W\in C^1(\mathbb{R}^m;\mathbb{R})$
\emph{nonnegative}, we have the weak monotonicity formula:
\begin{equation}\label{eqmonotoniMYWEAK}
\frac{d}{dR}\left(\frac{1}{R^{n-2}}\int_{B_R}^{}\left\{\frac{n-2}{2}|\nabla
u|^2+ nW(u) \right\}dx\right)\geq 0,\ \ R>0.
\end{equation}
In addition, if $u$ satisfies Modica's gradient bound
(\ref{eqmodica}), we have the strong monotonicity formula:
\begin{equation}\label{eqmonotoniMYstrong}
\frac{d}{dR}\left(\frac{1}{R^{n-1}}\int_{B_R}^{}\left\{\frac{n-2}{2}|\nabla
u|^2+ nW(u) \right\}dx\right)\geq 0,\ \ R>0.
\end{equation}
\end{theorem}
\begin{proof} By means of a direct calculation, it was shown in
\cite{alikakosBasicFacts} that, for solutions $u$ to (\ref{eqEq}),
the stress energy tensor $T(u)$, which is defined as the $n\times n$
matrix with entries
\[
T_{ij}=u_{,i}\cdot u_{,j}- \delta_{ij}\left(\frac{1}{2}|\nabla u|^2+
W\left(u\right) \right),\ \ i,j=1,\cdots,n,\ (\textrm{where}\
u_{,i}=u_{x_i}),\]
 satisfies
\begin{equation}\label{eqdiv}
\textrm{div}T(u)=0,
\end{equation}
using the notation $T=(T_1,T_2,\cdots,T_n)^\top$ and
$\textrm{div}T=(\textrm{div}
T_1,\textrm{div}T_2,\cdots,\textrm{div}T_n)^\top$, (see also
\cite{serfaty}). Observe that
\begin{equation}\label{eqtrace}
\textrm{tr}T=-\left(\frac{n-2}{2}|\nabla u|^2+ nW(u) \right),
\end{equation}
and that
\begin{equation}\label{eqpositiv}
T+\left(\frac{1}{2}|\nabla u|^2+ W\left(u\right) \right) I_n=(\nabla
u)^\top(\nabla u)\geq 0 \ \ \ \textrm{(in the matrix sense)},
\end{equation}
where $I_n$ stands for the $n\times n$ identity matrix.

As in \cite{schoen}, writing $x=(x_1,\cdots,x_n)$, and making use of
(\ref{eqdiv}), we calculate that
\begin{equation}\label{eq1}
\sum_{i,j=1}^{n}\int_{B_R}^{}\left(x_i
T_{ij}\right)_{,j}dx=\sum_{i,j=1}^{n}\int_{B_R}^{}\left\{\delta_{ij}T_{ij}+
x_{i}(T_{ij})_{,j}\right\}dx=\sum_{i=1}^{n}\int_{B_R}^{} T_{ii}dx.
\end{equation}
On the other side, from the divergence theorem, denoting $\nu=x/R$,
and making use of (\ref{eqpositiv}), we find that
\begin{equation}\label{eq2}
\sum_{i,j=1}^{n}\int_{B_R}^{}\left(x_i
T_{ij}\right)_{,j}dx=R\sum_{i,j=1}^{n}\int_{\partial B_R}^{}\nu_i
T_{ij} \nu_j dS\geq -R\int_{\partial B_R}^{}\left(\frac{1}{2}|\nabla
u|^2+ W\left(u\right) \right)dS.
\end{equation}
Since $W$ is nonnegative, if $n\geq 3$, we have that
\begin{equation}\label{eqineq1}
\frac{1}{2}|\nabla u|^2+ W\left(u\right)\leq \frac{1}{n-2}
\left(\frac{n-2}{2}|\nabla u|^2+ n W\left(u\right) \right).
\end{equation}

Let
\[
f(R)=\int_{B_R}^{}\left(\frac{n-2}{2}|\nabla u|^2+ n W\left(u\right)
\right) dx,\ \ R>0.
\]
By combining (\ref{eqtrace}), (\ref{eq1}), (\ref{eq2}) and
(\ref{eqineq1}), for $n\geq 3$, we arrive at
\[
-f (R)\geq -\frac{R}{n-2}\frac{d}{dR}f(R),\ \ R>0,
\]
which implies that
\[
\frac{d}{dR}\left(R^{2-n}f(R) \right)\geq 0,\ \ R>0,
\]
(clearly this also holds for $n=2$). We have thus shown the first
assertion of the theorem.

Suppose that $u$ additionally satisfies Modica's gradient bound
(\ref{eqmodica}). Then, we can strengthen (\ref{eqineq1}), for
$n\geq 2$, by noting that
\[
\frac{1}{2}|\nabla u|^2+ W\left(u\right)= \frac{1}{n-1}
\left(\frac{n-2}{2}|\nabla u|^2+\frac{1}{2}|\nabla u|^2+ (n-1)
W\left(u\right) \right)\leq \frac{1}{n-1} \left(\frac{n-2}{2}|\nabla
u|^2+ n W\left(u\right) \right).
\]
Now, by combining (\ref{eqtrace}), (\ref{eq1}), (\ref{eq2}) and the
above relation, we arrive at
 \[ -f (R)\geq -\frac{R}{n-1}\frac{d}{dR}f(R),\ \ R>0,
\]
which implies that
\[
\frac{d}{dR}\left(R^{1-n}f(R) \right)\geq 0,\ \ R>0,
\]
as desired.
\end{proof}
\begin{rem}
The weak monotonicity formula (\ref{eqmonotoniMYWEAK}) for the special case of the
Ginzburg-Landau system was stated (without proof) in
\cite{farinatwores}. If $m=1$ and $n=2$, the strong monotonicity
formula (\ref{eqmonotoniMYstrong}) was proven recently, by different
and intrinsically two dimensional techniques, in \cite{smyrnelis}.
\end{rem}
\begin{rem}
Mingfeng Zhao \cite{zhao} kindly informed me that all of the
monotonicity formulas in this paper can also be derived from
Pohozaev's identities for systems (see \cite{serfaty} for the case
of the Ginzburg-Landau system).
\end{rem}

\section{On a maximum principle for vector minimizers to the Allen-Cahn
energy}\label{appMAX}

In the recent paper \cite{alikakosPreprint}, the authors proved the
following variational maximum principle:

\begin{theorem}\label{thmAF}
Let $W:\mathbb{R}^m\to \mathbb{R}$ be $C^1$ and \emph{nonnegative}. Assume
that $W(a)=0$ for some $a\in \mathbb{R}^m$ and that there is $r_0>0$
such that (\ref{eqmonot}) holds. Let $\Omega \subset \mathbb{R}^n$
be an open, connected, bounded set, with $\partial \Omega$ minimally
smooth (Lipschitz continuous is enough), and suppose that $u\in
W^{1,2}(\Omega;\mathbb{R}^m)\cap L^\infty(\Omega;\mathbb{R}^m)$ is
minimal, in the sense that (\ref{eqminimal}) is satisfied.

If there holds
\begin{equation}\label{eqbdryAF}
|u(x)-a|\leq r\ \ \textrm{for}\ x\in \partial \Omega,
\end{equation}
for some $r\in \left(0,\frac{r_0}{2}\right)$, then it also holds
that
\begin{equation}\label{eqAFassert}
|u(x)-a|\leq r\ \ \textrm{for}\ x\in  \Omega.
\end{equation}
\end{theorem}
The main idea of the proof is that if the assertion is violated at
some point, then one can construct a suitable  competitor function
which agrees with $u$ on $\partial \Omega$ and has \emph{strictly less}
energy, which is impossible.

In this appendix, under the slight additional regularity assumption
that $W\in C_{loc}^{1,1}$ (which is consistent with most
applications), we will show that one can conclude just by showing
that the aforementioned competitor function has \emph{less or equal}
energy. Our main observation is to apply the unique continuation
principle for linear elliptic systems (see \cite{lopez} for other
applications). As a result, under the slight additional assumption that
$W\in C_{loc}^{1,1}$, we can simplify the corresponding proof in
\cite{alikakosPreprint}. Moreover, we can allow for the functions in
(\ref{eqmonot}) to be merely nondecreasing which is crucial for
establishing the second assertion of Case (b) in Theorem
\ref{thmLiouva}.  More precisely, we have the following theorem.

\begin{theorem}\label{thmMAXMY}
Assume that $W:\mathbb{R}^m\to \mathbb{R}$ is $C^{1,1}_{loc}$,
nonnegative, such that $W(a)=0$ for some $a\in \mathbb{R}^m$ and
that the functions in (\ref{eqmonot}) are nondecreasing. Moreover,
assume that
\[
W(u)>0 \ \ \textrm{if}\ \ |u-a|<2r_0\ \ \textrm{and}\ \ u\neq a.
\]
Then, the assertion of   Theorem \ref{thmAF} remains true.
\end{theorem}
\begin{proof}
Firstly, by standard elliptic regularity theory, we have that $u$ is
a smooth solution to the elliptic system in (\ref{eqEq}) in $\Omega$ and continuous up to the boundary (under
reasonable assumptions on $\partial \Omega$). Without loss of
generality, we take $a=0$. As in \cite{alikakosPreprint}, we set
\[
\rho(x)=|u(x)|\ \
\textrm{in} \ \ \Omega\ \ \textrm{and}\ \ \nu(x)=\frac{u(x)}{\rho(x)}
\ \
\textrm{in} \ \ \Omega_+=\{x\in \Omega\ :\ \rho>0 \}.
\]
 We also set
$\Omega_0=\{x\in \Omega\ :\ \rho=0 \}$ (actually, it can be shown that $\Omega_0=\emptyset$ but it is not important for the proof). It has been shown in
\cite{alikakosPreprint} that the energy of $u$ equals
\[
E(u;\Omega)=\frac{1}{2}\int_{\Omega}^{}|\nabla
\rho|^2dx+\frac{1}{2}\int_{\Omega_+}^{}\rho^2 |\nabla \nu|^2
dx+\int_{\Omega}^{}W(\rho \nu)dx.
\]
Let
\[
\tilde{u}(x)=\left\{\begin{array}{ll}
                      \textrm{min}\left\{\rho(x),r \right\}\alpha\left(\rho(x) \right)\nu(x), & x\in \Omega_+\cap \{\rho<2r \}, \\
  &   \\
                      0, & x\in \Omega_0\cup \{\rho \geq 2r \},
                    \end{array}
 \right.
\]
where $\alpha(\cdot)$ is the auxiliary function
\[
\alpha(\tau)=\left\{\begin{array}{ll}
                      1, & \tau\leq r, \\
                        &   \\
                      \frac{2r-\tau }{r},  & r\leq \tau \leq 2 r, \\
                        &   \\
                      0, & \tau \geq 2 r.
                    \end{array}
 \right.
\]
It was shown in \cite{alikakosPreprint} that $\tilde{u}\in
W^{1,2}(\Omega;\mathbb{R}^m)\cap L^\infty(\Omega;\mathbb{R}^m)$ and
that its energy equals
\[
E(\tilde{u};\Omega)=\frac{1}{2}\int_{\Omega}^{}|\nabla
\tilde{\rho}|^2dx+\frac{1}{2}\int_{\tilde{\Omega}_+}^{}\tilde{\rho}^2
|\nabla \nu|^2 dx+\int_{\Omega}^{}W(\tilde{\rho} \nu)dx,
\]
where $\tilde{\rho}(x)=|\tilde{u}(x)|$ and $\tilde{\Omega}_+=\{x\in
\Omega\ :\ \tilde{\rho}>0 \}$. Note that, thanks to (\ref{eqbdryAF}), we have
\begin{equation}\label{eqcontra}
u=\tilde{u}\ \ \textrm{on}\ \ \partial \Omega\ \ \textrm{and}\ \
|\tilde{u}|\leq r\ \ \textrm{a.e.\ in}\ \ \Omega.
\end{equation}
It follows readily that
\[
E(\tilde{u};\Omega)\leq E(u;\Omega),
\]
see also the proof in \cite{alikakosPreprint}. Consequently,
$\tilde{u}$ is also a minimizer in $\Omega$ subject to the same boundary
conditions as $u$. Hence, the function $\tilde{u}$ is smooth and
satisfies
\[
\Delta \tilde{u}=\nabla W(\tilde{u})\ \ \textrm{in}\ \ \Omega.
\]

We are now set to show that assertion (\ref{eqAFassert}) holds.
Suppose, to the contrary, that
\begin{equation}\label{eqcontra2}
|u(x_0)|>r\ \ \textrm{for some}\ \ x_0\in \Omega.
\end{equation}
We will first exclude the case
\[
r\leq \rho(x)\leq 2r\ \ \textrm{for all}\ \ x\in \Omega.
\]
If not, the function
\[
\hat{u}=r\nu(x)\in W^{1,2}(\Omega;\mathbb{R}^m)\cap
L^\infty(\Omega;\mathbb{R}^m)
\]
would have strictly less energy then $u$ (because
$\int_{\Omega}^{}|\nabla \rho|^2dx>0$) while $\hat{u}=u$ on
$\partial \Omega$, which is impossible. Next, we exclude entirely
the case
\[
r\leq \rho(x),\ \ x\in \Omega.
\]
If not, there would exist $x_1\in \Omega$ such that $\rho(x_1)>2r$.
This implies that $\tilde{u}=0$ on a set of positive measure
containing $x_1$. Since $\nabla W(u)$ is locally Lipschitz
continuous, we see that $\tilde{u}$ satisfies the linear system
\[
\Delta \tilde{u}=Q(x)\tilde{u}\ \ \textrm{in}\ \ \Omega, \ \
\textrm{where} \ \ Q(x)=\int_{0}^{1}W_{uu}(t \tilde{u}) dt\ \
\textrm{is bounded in norm}.
\]
On the other hand, because $\tilde{u}=0$ on a set of positive
measure, by the unique continuation principle for linear elliptic
systems (see \cite{hormander}), we infer that $\tilde{u}\equiv 0$
which is clearly impossible (otherwise $|u|\geq 2r$ in $\Omega$).
Therefore, we may assume that there exists a set $D\subset \Omega$
with positive measure such that
\[
u=\tilde{u}\ \ \textrm{in}\ \ D.
\]
As before, by considering the linear system for the difference
$u-\tilde{u}$, we conclude that $\tilde{u}\equiv u$. We have thus
arrived at a contradiction, because of (\ref{eqcontra}) and
(\ref{eqcontra2}).
\end{proof}

\textbf{Acknowledgements.} Supported by the ``Aristeia'' program of the Greek
Secretariat for Research and Technology.

\end{document}